\documentclass[12pt]{article}

\usepackage{amssymb}%
\usepackage{amsmath}%
\usepackage{amsthm}%
\usepackage{url}
\usepackage{wasysym}
\usepackage{xcolor}%
\usepackage{mathrsfs}
\usepackage{bbm}
\usepackage{tikz}

\usetikzlibrary{arrows}
\usetikzlibrary{decorations.markings}
\usetikzlibrary {graphs,quotes}
\usepackage{tkz-graph}

\tikzset{
 vertex/.style={circle,draw,minimum size=1.5em},
 edge/.style={->,> = latex'}
}

\allowdisplaybreaks%

{\theoremstyle{plain}%
 \newtheorem{theorem}{Theorem}

}
{\theoremstyle{remark}
\newtheorem{remark}{Remark}
}
{\theoremstyle{definition}
\newtheorem{definition}{Definition}
\newtheorem{example}{Example}
}

\begin{document}

\begin{center}
{\Large A lift of chromatic symmetric functions to $\textsf{NSym}$}
\end{center}

\begin{center}
{\textsc{John M. Campbell}} 

 \ 

\end{center}

\begin{abstract}
 If we consider previously introduced extensions of Stanley's chromatic symmetric function $X_{G}(x_1, x_2, \ldots)$ for a graph $G$ to elements in the 
 algebra $\textsf{QSym}$ of quasisymmetric functions and in the algebra $\textsf{NCSym}$ of symmetric functions in noncommuting variables, this 
 motivates our introduction of a lifting of $X_{G}$ to the dual of $\textsf{QSym}$, i.e., the algebra $ \textsf{NSym}$ of noncommutative symmetric 
 functions, as opposed to $\textsf{NCSym}$. For an unlabelled directed graph $D$, our extension of chromatic symmetric functions provides an element 
 $ \text{{\bf X}}_{D}$ in $\textsf{NSym}$, in contrast to the analogue $Y_{G} \in \textsf{NCSym}$ of $X_{G}$ due to Gebhard and Sagan. Letting $G$ 
 denote the undirected graph underlying $D$, our construction is such that the commutative image of $ \text{{\bf X}}_{D} $ is $ X_{G}$. This projection 
 property is achieved by lifting Stanley's power sum expansion for chromatic symmetric functions, with the use of the $\Psi$-basis of $\textsf{NSym}$, so 
 that the orderings of the entries of the indexing compositions are determined by the directed edges of $D$. We then construct generating sets for 
 $ \textsf{NSym}$ consisting of expressions of the form $\text{{\bf X}}_{D}$, building on the work of Cho and van Willigenburg on chromatic generating 
 sets for $\textsf{Sym}$. 
\end{abstract}

\vspace{0.1in}

\noindent {\footnotesize \emph{MCS}: 05E05, 05C15}

\vspace{0.1in}


\section{Introduction}
 Chromatic symmetric functions were famously introduced by Stanley \cite{Stanley1995} via an extension of the chromatic polynomial construction due to 
 Birkhoff \cite{Birkhoff191213} and are of importance within many areas of algebraic combinatorics. This raises questions as to how mathematical objects of a 
 similar nature could be obtained in a noncommutative setting. In this regard, chromatic symmetric functions have been lifted to the algebra 
 $ \textsf{NCSym}$ of symmetric functions in noncommuting variables by Gebhard and Sagan \cite{GebhardSagan2001}. However, this algebra, which 
 dates back to the work of Wolf in 1936 \cite{Wolf1936} (cf.\ \cite{BergeronReutenauerRosasZabrocki2008}), is not equivalent to the algebra 
 $ \textsf{NSym}$ of noncommutative symmetric functions introduced by Gelfand et al.\ \cite{GelfandKrobLascouxLeclercRetakhThibon1995} in 1995. 
 Ideally, by using noncommutative symmetric functions to obtain a lifting of Stanley's extension $X_{G}$ of the chromatic polynomial for a graph $G$, this 
 could lead toward new ways of approaching open problems concerning chromatic symmetric functions, as in the Stanley--Stembridge conjecture and 
 Stanley's tree isomorphism conjecture. This is representative of a main source of motivation behind the purpose of this paper, which is to introduce and 
 apply a lifting to $\textsf{NSym}$ of chromatic symmetric functions. 

 The foregoing considerations raise questions as to what could be viewed as being most useful or appropriate when it comes to a 
 noncommutative version of $X_{G}$, in the context of a given problem on chromatic symmetric functions. This may be seen as being in parallel with the 
 active research on analogues in $\textsf{NSym}$ of the Schur basis of the algebra $\textsf{Sym}$ of (commutative) symmetric functions. Such research is, 
 in large part, motivated by how Schur-like bases of $\textsf{NSym}$ could be helpful in the development of deeper understandings of unsolved problems 
 related to the classical Schur basis. With the exception of a very recent preprint due to Wang and Zhou \cite{WangZhouunpublished}, it seems that 
 the problem of devising ``chromatic-like'' elements in or bases of $\textsf{NSym}$ has not previously been considered. We offer a new approach toward 
 this problem, and this new approach may be thought of as being based on the idea as to how a given digraph $D$, with $G$ as its underlying undirected 
 graph, can give rise, in a natural and useful way that depends on structural properties of $D$, to an element $\text{{\bf X}}_{D}$ in $\textsf{NSym}$ 
 such that 
\begin{equation}\label{desiredprojection}
 \chi(\text{{\bf X}}_{D}) = X_{G}, 
\end{equation}
 according to the forgetful morphism 
\begin{equation}\label{chimorphism}
 \chi\colon \textsf{NSym} \to \textsf{Sym}
\end{equation}
 providing the commutative image of a given element in $\textsf{NSym}$. This is in contrast to the work of Wang and Zhou \cite{WangZhouunpublished}, 
 which does not involve directed graphs. 

 The research interest, from both combinatorial and algebraic perspectives, in the desired projection property in \eqref{desiredprojection} may be 
 demonstrated in a number of different ways. Informally, if we think of the forgetful morphism $\chi$ as ``forgetting \emph{algebraic} information'' 
 about elements in its domain, then, since $X_{G}$ is in the codomain of $\chi$ for every simple graph $G$, this raises the question as to how we could also 
 think of $\chi$ as ``forgetting \emph{graph-theoretic} information.'' By formalizing this idea with the forgetful functor from the category of directed 
 graphs to the category of undirected graphs, this points toward the research interest in and inherently interdisciplinary nature of 
 \eqref{desiredprojection}. This can be further motivated if we consider \eqref{desiredprojection} in relation to projection properties satisfied by 
 noncommutative analogues of Schur functions, as described below. 

 Since $\textsf{Sym}$ is contained in the algebra $\textsf{QSym}$ of quasisymmetric functions dual to $\textsf{NSym}$, we find, by duality, that 
 $ \textsf{NSym}$ projects onto $\textsf{Sym}$, i.e., according to the algebra morphism in \eqref{chimorphism}. In view of the combinatorial and 
 representation-theoretic interest in how Schur functions can be generalized via the containment of $\textsf{Sym}$ in $\textsf{QSym}$, this points toward 
 corresponding interests in the dual notion as to how Schur-like elements in $\textsf{NSym}$ could project onto the Schur basis of $\textsf{Sym}$, via 
 \eqref{chimorphism}. In this regard, what have been referred to as the \emph{canonical Schur-like bases} of $\textsf{NSym}$ 
 all project onto the Schur basis of 
 $ \textsf{Sym}$ in a natural way. 
 These bases of $\textsf{NSym}$ are the \emph{dual quasi-Schur basis} \cite{BessenrodtLuotovanWilligenburg2011}, 
 the \emph{immaculate basis} \cite{BergBergeronSaliolaSerranoZabrocki2014}, and the \emph{shin basis} \cite{CampbellFeldmanLightShuldinerXu2014}, 
 and each of these bases may be seen as having different advantages in terms of better or closer approximations of different properties of Schur functions, 
 in a noncommutative context. As something of a graph-theoretic counterpart to this phenomenon, different ways of lifting to $\textsf{NSym}$ the family 
 of expressions of the form $X_{G} \in \textsf{Sym}$ can offer different ways of ``capturing'' or approximating combinatorial and algebraic 
 properties of $X_{G}$. 

 The foregoing considerations give weight to the research interest in the lifting we introduce of chromatic symmetric functions to $\textsf{NSym}$. 
 Informally, our construction is based on the idea as to how an unlabelled directed graph $D$, with $G$ as its underlying undirected graph, naturally gives 
 rise to permutations of the indices arising in Stanley's expansion formula for $X_{G}$ in the power sum basis, so as to form an element $\text{{\bf X}}_{D} 
 \in \textsf{NSym}$ such that the projection identity in \eqref{desiredprojection} holds. 

\section{Preliminaries and background}
 Let $\lambda$ denote an integer partition, i.e., a tuple of positive integers that is weakly decreasing, writing $\lambda = (\lambda_{1}, \lambda_{2}, 
 \ldots, \lambda_{\ell(\lambda)})$. 
 The \emph{complete homogeneous generator} $h_{n}$ 
 is given by the formal sum of expressions of the form 
 $x_{i_{1}} x_{i_{2}} \cdots x_{i_{n}}$ 
 for positive integer indices $ i_{1} \leq i_{2} \leq \cdots \leq i_{n}$, 
 writing $h_{0} = 1$, 
 and writing 
\begin{equation}\label{hlambda}
 h_{\lambda} = h_{\lambda_{1}} h_{\lambda_{2}} \cdots h_{\lambda_{\ell(\lambda)}} 
\end{equation}
 to denote \emph{complete homogeneous symmetric functions}, with $h_{()} = 1$ 
 for the empty partition $()$. We may define 
\begin{equation}\label{Symhgenerators}
 \textsf{Sym} := \mathbbm{k}[h_{1}, h_{2}, \ldots], 
\end{equation}
 i.e., so that $\textsf{Sym}$ is the free commutative $\mathbbm{k}$-algebra with $\{ h_{1}, h_{2}, \ldots \}$ as 
 a generating set, for a fixed field $\mathbbm{k}$ (which is understood to be of characteristic 0). As suggested in 
 \eqref{hlambda}, it is customary to let the bases of $\textsf{Sym}$ be indexed by the set $\mathcal{P}$ of integer partitions. 

 We endow $\textsf{Sym}$ with a graded algebra structure by setting the degree of 
 $h_{n}$ as $n$, i.e., so that $\textsf{Sym}$ may 
 be understood as the free commutative $\mathbbm{k}$-algebra that has one generator in every degree, referring to the classic text by Macdonald 
 for a way of formalizing this equivalence via an inverse limit construction \cite[p.\ 18]{Macdonald1995}. This can also be used to formalize the construction 
 of the power sum basis of $\textsf{Sym}$. In this direction, we set the \emph{power sum generator} $p_{n} \in \textsf{Sym}$ as the formal sum 
 of expressions of the form $x_{i}^{n}$ for $i \in \mathbb{N}$ and indeterminates $x_{i}$. This leads us to write $p_{\lambda} = p_{\lambda_{1}} 
 p_{\lambda_{2}} \cdots p_{\lambda_{\ell(\lambda)}}$, leading us to define the \emph{power sum basis} $\{ p_{\lambda} \}_{\lambda \in \mathcal{P}}$. 

 By direct analogy with \eqref{Symhgenerators}, we set 
\begin{equation}\label{NSymdefinition}
 \textsf{NSym} := \mathbbm{k}\langle H_{1}, H_{2}, \ldots \rangle 
\end{equation}
 for noncommuting indeterminates of the form $H_{n}$ such that $\text{deg}(H_{n}) = n$, i.e., so that \eqref{NSymdefinition} may be understood as the 
 free $\mathbbm{k}$-algebra with one generator in each degree. By analogy with $\textsf{Sym}$, it is standard to let bases of $\textsf{NSym}$ be 
 indexed by the set $ \mathcal{C} $ of integer \emph{compositions}, i.e., tuples of positive integers. For $\alpha \in \mathcal{C}$, writing $\alpha = 
 (\alpha_{1}, \alpha_{2}, \ldots, \alpha_{\ell(\alpha)})$, we write $H_{\alpha} = H_{\alpha_{1}} H_{\alpha_{2}} \cdots H_{\alpha_{\ell(\alpha)}}$, with 
 $H_{()} = 1$. This is by direct analogy with the commutative product in \eqref{hlambda}, with $\{ H_{\alpha} \}_{\alpha \in \mathcal{C}}$ forming a 
 basis of $\textsf{NSym}$ that we may refer to as the \emph{complete homogeneous basis} of $\textsf{NSym}$. 

 The algebra morphism $\chi$ in \eqref{chimorphism} may be defined by setting $\chi(H_{\alpha}) = h_{\text{sort}(\alpha)}$, and by extending this relation 
 linearly, where $\text{sort}(\alpha)$ denotes the integer partition obtained by sorting the entries of a given composition $\alpha$. Our construction 
 requires a basis $\{ \Psi_{\alpha} \}_{\alpha \in \mathcal{C}}$ of $\textsf{NSym}$ defined below that projects onto the basis $\{ p_{\lambda} 
 \}_{\lambda \in \mathcal{P}}$. 

 We let the \emph{refinement ordering} be denoted as 
 $\leq$ and defined on tuples of positive integers so that 
 $a \leq b$ if $D(b) \subseteq D(a)$, where $D(c)$ denotes the 
 descent set of a composition $c$, writing $D(c) = \{ c_{1}, c_{1} + c_{2}, \ldots, c_{1} + c_{2} + \cdots + c_{\ell(c) - 1} \} $. The 
 \emph{ribbon basis} $\{ 
 R_{\alpha} \}_{\alpha \in \mathcal{C}}$ of $\textsf{NSym}$ may be defined in relation to the $H$-basis according to the equivalent relations such that 
\begin{align*}
 R_{\alpha} = \sum_{\beta \geq \alpha} (-1)^{\ell(\alpha) - \ell(\beta)} H_{\beta} \ \ \ \text{and} \ \ \ 
 H_{\alpha} = \sum_{\beta \geq \alpha} R_{\beta}. 
\end{align*}
 This leads us to define $ \Psi_{n} = \sum_{i=0}^{n-1} (-1)^{i} R_{(1^{i}, n - i)} $ for $n \in \mathbb{N}$, writing 
\begin{equation}\label{Psiprodrule}
 \Psi_{\alpha} = \Psi_{\alpha_{1}} 
 \Psi_{\alpha_{2}} \cdots \Psi_{\alpha_{\ell(\alpha)}}
\end{equation}
 for $\alpha \in \mathcal{C}$, with $\Psi_{()} = 1$. 
 The basis $\{ \Psi \}_{\alpha \in \mathcal{C}}$ 
 of $\textsf{NSym}$
 provides a noncommutative version of the power sum basis in the sense that 
\begin{equation}\label{chiPsin}
 \chi(\Psi_{n}) = p_{n}. 
\end{equation}

 Letting $G = (V(G), E(G))$ be an undirected graph, the \emph{chromatic symmetric function} associated with $G$ may be denoted and defined so that 
\begin{equation}\label{XGdefinition}
 X_{G}(x_{1}, x_{2}, \ldots) = X_{G} = \sum_{\text{$\kappa$ proper}} \prod_{v \in V(G)} x_{\kappa(v)}, 
\end{equation}
 where the above sum is over all proper colorings $\kappa\colon V(G) \to \mathbb{N}$, referring to \cite{Stanley1995} and subsequent literature for details. 
 The definition in \eqref{XGdefinition} generalizes the chromatic polynomial $\chi_{G}$ for $G$ in the sense that $X_{G}$ evaluated with the arguments 
 $x_{1} = x_{2} = \cdots = x_{n} = 1$ and the arguments $x_{n+1} = x_{n+2} = \cdots = 0$ equals $\chi_{G}(n)$, again with reference to 
 Stanley's work \cite{Stanley1995}. 

\subsection{Previous extensions of chromatic symmetric functions}
 \emph{Chromatic quasisymmetric functions} were introduced in \cite{ShareshianWachs2016} and are defined on labeled digraphs; see also notable 
 research contributions on chromatic quasisymmetric functions as in \cite{AlexanderssonPanova2018}. 
 By direct analogy with Stanley's definition for $X_{G}$, as in \eqref{XGdefinition}, 
 Gebhard and Sagan \cite{GebhardSagan2001} set 
\begin{equation}\label{YGdefinition}
 Y_{G} := \sum_{\kappa} y_{\kappa(v_{1})} y_{\kappa(v_{2})} \cdots y_{\kappa(v_{d})} 
\end{equation}
 where the sum in \eqref{YGdefinition} is over all proper colorings of a graph $G$ with vertices labeled with $v_{1}$, $v_{2}$, $\ldots$, $v_{d}$, 
 where the $y$-variables in \eqref{YGdefinition} are set to be noncommuting, in contrast to \eqref{XGdefinition}. The noncommutative chromatic 
 symmetric function defined in \eqref{YGdefinition} is in the algebra $\textsf{NCSym}$ reviewed below \cite{DahlbergvanWilligenburg2020}. It is 
 explicitly stated by Gebhard and Sagan \cite{GebhardSagan2001} that symmetric functions in noncommuting \emph{variables} are not the same as 
 the noncommutative symmetric functions introduced by Gelfand et al.\ \cite{GelfandKrobLascouxLeclercRetakhThibon1995}, and this is further clarified 
 below and motivates a lifting of chromatic symmetric functions to $\textsf{NSym}$. 

 The algebra $\textsf{NCSym}$ of symmetric functions in noncommuting variables was introduced by Bergeron et al.\ in 
 \cite{BergeronReutenauerRosasZabrocki2008}, building on the work of Wolf \cite{Wolf1936} on symmetric functions in noncommuting variables. In 
 contrast to bases of $\textsf{NSym}$ being indexed by integer compositions, the bases of $\textsf{NCSym}$ are indexed by set partitions. This can be 
 made more explicit by writing 
\begin{align}
 \textsf{NCSym} & = \bigoplus_{n \geq 0} \mathscr{L}\{ \text{{\bf m}}_{A} \}_{A \vdash [n]} \ \ \ \text{and} \nonumber \\ 
 \textsf{NSym} & = \bigoplus_{n \geq 0} \mathscr{L}\{ H_{\alpha} \}_{\alpha \vDash n}, \label{NSymgraded}
\end{align} 
 where $A$ denotes a set partition of $[n] = \{ 1, 2, \ldots, n \}$, and where each family of the form $\{ \text{{\bf m}}_{A} \}_{A \vdash [n]} $ is a basis 
 of the graded component of $\textsf{NCSym}$ of degree $n$, and where $\alpha \vDash n$ indicates that $\alpha$ is a composition of $n$, i.e., that 
 the sum of the entires of $\alpha$ equals $n$, letting it be understood that the empty composition $()$ is such that $() \vDash 0$. 

 Our introduction of and analyses of chromatic noncommutative symmetric functions are motivated by past generalizations and variants of chromatic 
 symmetric functions, and this includes chromatic multisymmetric functions \cite{CrewHaithcockReynesSpirkl2024}, a $K$-theoretic analogue of 
 chromatic symmetric functions \cite{CrewPechenikSpirkl2023}, what are referred to as \emph{$H$-chromatic symmetric functions} defined for a given 
 graph $G$ and with a parameter given by a graph $H$ such that $|V(G)| = |V(H)|$ \cite{EaglesFoleyHuangKarangozishviliYu2022}, an extension of 
 $X_{G}$ to binary delta-matroids \cite{NenashevaZhukov2021}, an analogue of chromatic symmetric functions used to construct bases for the algebra 
 generated by the Schur $Q$-functions \cite{ChoHuhNam2021}, extended chromatic symmetric functions for vertex-weighted graphs 
 \cite{CrewSpirkl2020}, and the \emph{generalized chromatic functions} that are due to Aliniaeifard et al.\ and that generalize chromatic symmetric 
 functions via $P$-partitions \cite{AliniaeifardLivanWilligenburg2024}.

 A notable generalization of chromatic quasisymmetric functions based on oriented graphs is due to Ellzey \cite{Ellzey2017}. This suggests that 
 generalizing chromatic symmetric functions to elements in the algebra dual to $\textsf{QSym}$ would provide a natural follow-up to the work 
 Ellzey \cite{Ellzey2017}, in which noncommutative symmetric functions were not considered. Ellzey's construction is such that the case whereby a 
 directed graph is acyclic reduces to Shareshian's and Wachs' chromatic quasisymmetric functions \cite{ShareshianWachs2016}, 
 and this is not equivalent to our new approach toward lifting 
 $X_{G} \in \textsf{Sym}$ to $\textsf{NSym}$, 
 which does not seem to have been considered in the extant literature subsequent and related to 
 Ellzey's generalization of chromatic quasisymmetric functions 
 \cite{AlexanderssonPanova2018,AlexanderssonSulzgruber2021,AliniaeifardLivanWilligenburg2024,AlistePrietoCrewSpirklZamora2021,
CrewSpirkl2020,Dahlberg2020,EllzeyWachs2020,LoehrWarrington2024,WangWang2020,White2021}. 

\section{Main construction and results}
 The key to our lifting of $X_{G}$ is given by 
 a formula due to Stanley  for expanding $X_{G}$ into the power sum basis \cite{Stanley1995}, namely 
\begin{equation}\label{XGtop}
 X_{G} = \sum_{S \subseteq E} (-1)^{|S|} p_{\lambda(S)}, 
\end{equation}
 where the expression $\lambda(S)$ involved in \eqref{XGtop} may be defined as follows. For a subset $S$ of $E = E(G)$ for a graph $G = (V(G), E(G))$ 
 such that $|V(G)| = n$, the expression $\lambda(S)$ denotes the partition of $n$ obtained as follows. By taking the spanning subgraph of $G$ that 
 has $V$ as its vertex set and $S$ as its edge set, we then form a tuple such that its entries list, in a given order, the numbers of vertices in the connected 
 components of the spanning subgraph we have specified. Then $\lambda(S)$ is obtained by sorting the entries of the specified tuple. 

\begin{example}
 Inputting 
\begin{verbatim}
G = graphs.CycleGraph(4);
XG = G.chromatic_symmetric_function(); 
XG
\end{verbatim}
 into {\tt SageMath}, we find that the chromatic symmetric function $X_{G}$, 
 for the case whereby $G$ is the cycle graph of order $4$, 
 is such that 
\begin{equation}\label{XGtopSageC4}
 X_{G} = p_{1111} - 4 p_{211} + 2 p_{22} + 4 p_{31} - 3 p_{4}. 
\end{equation}
\end{example}

 In view of our goal of lifting \eqref{XGtop} to expansions in $\textsf{NSym}$ defined for directed graphs, we begin by recalling 
 some basic terminology concerning digraphs, 
 with reference to the text by Chartrand and Lesniak \cite[\S1.4]{ChartrandLesniak2005}. A \emph{directed 
 graph} or \emph{digraph} is an ordered pair $D = (V(D), E(D))$ for a finite set $V(D)$ and for a set $E(D)$ consisting of ordered pairs consisting of 
 elements of $V(D)$, with the elements in $E(D)$ being referred to as \emph{arcs} or as \emph{directed edges}. 
 Of central importance in our construction are the concepts of \emph{outdegree} and \emph{indegree}, again with reference to Chartrand and Lesniak's 
 text \cite[p.\ 26]{ChartrandLesniak2005}. For an arc $a = (u, v)$ in a digraph $D$, the vertex $u$ is referred to as being \emph{adjacent to} $v$ 
 and $v$ is said to be \emph{adjacent from} $u$. The \emph{outdegree} $\text{od} \, w$ of a vertex in a digraph $D$ is the number of vertices of $D$ 
 adjacent from $w$. Similarly, the \emph{indegree} $\text{id} \, w$ of the vertex $w$ is equal to the number of vertices of $D$ adjacent to $v$. The 
 \emph{degree} $\text{deg} \, w$ is such that $\text{deg} \, w = \text{od} \, w + \text{id} \, w$. 

\subsection{Construction and well-definedness}
  Our lifting $X_{G}$ so as to construct generating sets for $\textsf{NSym}$ indexed by digraphs   is closely related to what is referred to as the \emph{First  
  Theorem   of Digraph Theory} \cite[p.\ 26]{ChartrandLesniak2005},  
 which gives us, for a digraph $D$ with a vertex set we denote with $V(D) = \{ v_{1}, v_{2}, \ldots, v_{|V(D)|} \}$, that 
 $$ \sum_{i=1}^{|V(D)|} \text{od} \, v_{i} = \sum_{i=1}^{|V(D)|} \text{id} \, v_{i} = |E(D)|. $$ 

\begin{definition}
 For a subset $W = \{ w_{1}, w_{2}, \ldots, w_{|W|} \}$ of $V = V(D)$ for a digraph $D$, 
 we define the \emph{total degree} $\text{td} \, W$ of $W$ as 
 $$ \text{td} \, W := \sum_{i=1}^{|W|} \text{od} \, w_{i} 
 - \sum_{i=1}^{|W|} \text{id} \, w_{i}. $$ 
\end{definition}

\begin{example}\label{examplereversed}
 Let $D = (V(D), E(D))$ denote the digraph illustrated below, writing 
 $V = V(D) = \{ v_{1}, v_{2}, v_{3}, v_{4} \}$. 
\begin{center}
\begin{tikzpicture}
\node[vertex] (1) at (0,0) {$v_1$};
\node[vertex] (2) at (2,0) {$v_2$};
\node[vertex] (3) at (0,-2) {$v_3$};
\node[vertex] (4) at (2,-2) {$v_4$};
\draw[black, very thick, edge] (1) -- (2); 
\draw[black, very thick, edge] (3) -- (1);
\draw[black, very thick, edge] (2) -- (4);
\draw[black, very thick, edge] (3) -- (4);
\end{tikzpicture}
\end{center}
 Now, we fix the subset $S = \{ (v_{1}, v_{2}), 
 (v_{2}, v_{4}) \}$ of $E = E(D)$ highlighted below. 
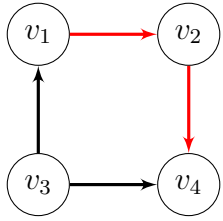
\begin{center}
\begin{tikzpicture}
\node[vertex] (1) at (0,0) {$v_1$};
\node[vertex] (2) at (2,0) {$v_2$};
\node[vertex] (3) at (0,-2) {$v_3$};
\node[vertex] (4) at (2,-2) {$v_4$};
\draw[red, very thick, edge] (1) -- (2); 
\draw[black, very thick, edge] (3) -- (1);
\draw[red, very thick, edge] (2) -- (4);
\draw[black, very thick, edge] (3) -- (4);
\end{tikzpicture}
\end{center}
 By then writing $S' = \{ \{ v_{1}, v_{2} \}, \{ v_{2}, v_{4} \} \}$ to denote the edge subset corresponding to $S$ by replacing directed edges with 
 corresponding edges of the undirected graph underlying $D$, we find that $\lambda(S') = (3, 1) \vdash |V|$. We find that the total degrees of the 
 subsets $\{ v_{1}, v_{2}, v_{4} \}$ and $\{ v_{3} \}$ are $-2$ and $2$, respectively. 
\end{example}

\begin{definition}\label{componenttuple}
 Let $D = (V(D), E(D))$ be a digraph and let $G = (V(G), E(G))$ denote its underlying undirected graph, writing $V = V(D) = V(G)$. Moreover,  we let the 
 vertices of $D$ and $G$ be labeled and ordered,  writing  
\begin{equation}\label{displayVorder}
 V = \{ v_{1} < v_{2} < \cdots < v_{|V|} \}. 
\end{equation}
    Let $S \subseteq  E(D)$. We  also let $S'$ denote  
  the set of edges in $E(G)$  obtained by forming $2$-sets from the arcs in $S$.  We write 
\begin{equation}\label{tuplesuperscript}
 \left( C_{1}^{S, D}, C_{2}^{S, D}, \ldots, C^{S, D}_{ \ell\left( \lambda(S') \right) } \right) 
 \end{equation}
 to denote an ordered tuple   consisting   of all of the connected components of the spanning subgraph of $G$ induced by $S'$, where the ordering in 
 \eqref{tuplesuperscript} is determined as follows, and where we may write $C_{i} = C_{i}^{S, D}$. For each index $i \in \{ 1, 2, \ldots, \ell\left( 
 \lambda(S') \right) \}$, the cardinality of $C_{i}$ satisfies $|C_{i}| = \left( \lambda(S') \right)_{i}$. If $|C_{i}| = |C_{i+1}|$ for $i \in \{ 1, 2, 
 \ldots, \ell\left( \lambda(S') \right) - 1 \}$, then we sort $C_{i}$ and $C_{i+1}$ by identifying a subset of $V$ as a word on the alphabet $\{ 1, 2, 
 \ldots \}$ and by then sorting words in reverse lexicographic order. We refer to the tuple in \eqref{tuplesuperscript} as the \emph{component tuple} 
 of $D$ induced by $S$ and $<$. 
\end{definition}

\begin{example}\label{ex2ndcycle}
 We set $D = (V(D), E(D))$ denote the digraph illustrated below, again writing 
 $V = V(D) = \{ v_{1} < v_{2} < v_{3} < v_{4} \}$, 
 noting the contrast to Example \ref{examplereversed}. 
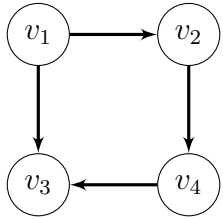
\begin{center}
\begin{tikzpicture}
\node[vertex] (1) at (0,0) {$v_1$};
\node[vertex] (2) at (2,0) {$v_2$};
\node[vertex] (3) at (0,-2) {$v_3$};
\node[vertex] (4) at (2,-2) {$v_4$};
\draw[black, very thick, edge] (1) -- (2); 
\draw[black, very thick, edge] (1) -- (3);
\draw[black, very thick, edge] (2) -- (4);
\draw[black, very thick, edge] (4) -- (3);
\end{tikzpicture}
\end{center}
 Now, we fix the subset $S = \{ (v_{1}, v_{3}) \}$ of $E = E(D)$ highlighted below. 
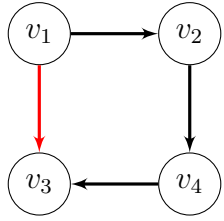
\begin{center}
\begin{tikzpicture}
\node[vertex] (1) at (0,0) {$v_1$};
\node[vertex] (2) at (2,0) {$v_2$};
\node[vertex] (3) at (0,-2) {$v_3$};
\node[vertex] (4) at (2,-2) {$v_4$};
\draw[black, very thick, edge] (1) -- (2); 
\draw[red, very thick, edge] (1) -- (3);
\draw[black, very thick, edge] (2) -- (4);
\draw[black, very thick, edge] (4) -- (3);
\end{tikzpicture}
\end{center}
 Being consistent with the notation in Definition \ref{componenttuple}, we have that $S' $ $ =$ $ \{ \{ v_{1}$, $ v_{3} \} \} $ $ \subseteq $ $ E(G)$. We 
 thus have that $\lambda(S') = (2, 1, 1)$. According to the given labeling for $D$, we find that the component tuple of $D$ (induced by $S$ and $<$) is 
 $ ( \{ v_{1}, v_{3} \}, \{ v_{4} \}, \{ v_{2} \} )$. 
\end{example}

\begin{definition}\label{definitionalpha}
 Let $D$ and $G$  be as in Definition  \ref{componenttuple}, 
  letting  $V$   be ordered as in  \eqref{displayVorder}.
  Moreover, we let $S$ and $S'$
    be as in Definition \ref{componenttuple}.  
   We set $\alpha(S) = \alpha^{>}_{D}(S)$ as an integer composition of $|V|$ 
 obtained by permuting the entries of $\lambda(S')$, writing 
\begin{equation}\label{alphaSlambdaorder}
 \alpha(S) = \left( 
 (\lambda(S'))_{i_{1}^{S, D, >}}, 
 (\lambda(S'))_{i_{2}^{S, D, >}}, \ldots, 
 (\lambda(S'))_{i_{\ell(\lambda(S'))}^{S, D, >}} \right), 
\end{equation}
 where the indices of the form $i_{j}$ for $j \in \{ 1, 2, \ldots, \ell(\lambda(S')) \}$ in \eqref{alphaSlambdaorder} 
 may be determined as follows (and are defined according to the order relation $>$ in \eqref{displayVorder}), 
 writing $i_{j} = i_{j}^{S, D, >}$. 
 Letting $j, k \in \{ 1, 2, \ldots, \ell(\lambda(S')) \}$, we let the relation $\rhd = \rhd^{>}$ (which again depends
 on the ordering $>$ in \eqref{displayVorder})
 be such that $ i_{j} \rhd i_{k} $ if: 
\begin{enumerate}

\item $ \text{td} \, C_{j}^{S, D} > \text{td} \, C_{k}^{S, D}$; or 

\item $ \text{td} \, C_{j}^{S, D} = \text{td} \, C_{k}^{S, D} $ and 
 $ | C_{j}^{S, D} | > | C_{k}^{S, D} |$; or 

\item $ | C_{j}^{S, D} | = | C_{k}^{S, D} |$ 
 and $ C_{j}^{S, D} $ is lexicographically strictly greater than $ C_{k}^{S, D}$, i.e., 
 according to the ordering $v_{1} < v_{2} < \cdots < v_{|V|}$. 

\end{enumerate}

\noindent We may define $\unrhd = \unrhd^{>}$, 
 by letting $ i_{j} \unrhd i_{k}$ 
 if it is either the case that $ i_{j} \rhd i_{k} $ or that $ i_{j} = i_{k}$. 
\end{definition}

\begin{remark}
 It is immediate that: For any order $>$ such that the reflexive closure of $>$ is a total order on $V$, the relation $ \unrhd^{>}$ is a total order. 
\end{remark}

\begin{example}
 For $D$ and $S$ as in Example \ref{examplereversed}, 
 we recall that 
 $\lambda(S') = (3, 1)$, with 
 $\text{td} \, \{ v_{1}, v_{2}, v_{4} \} = -2$ and 
 $\text{td} \, \{ v_{3} \} = 2$. 
 We then find that 
\begin{align*}
 \alpha(S) 
 & = \left( (\lambda(S'))_{i_{1}^{S, D}}, (\lambda(S'))_{i_{2}^{S, D}} \right) \\ 
 & = \left( (\lambda(S'))_{2}, (\lambda(S'))_{1} \right) 
 = (1, 3), 
\end{align*}
 since $ i_{2}^{S, D} \rhd i_{1}^{S, D}$, since 
 $\text{td} \, \{ v_{3} \} = 2 > \text{td} \, \{ v_{1}, v_{2}, v_{4} \} = -2$. 
\end{example}

\begin{theorem}\label{theoremalphadefined}
 For a digraph $D$ with $G$ as its underlying undirected graph, and for $S \subseteq E(D)$, any two labelings of $V(D) = V(G)$ produce the same 
 tuple $\alpha(S)$. 
\end{theorem}

\begin{proof}
 Being consistent with the notation in Definition \ref{definitionalpha}, suppose that the ordering $V = \{ v_{1} < v_{2} < \cdots < v_{|V|} \}$ 
 produces the same composition in \eqref{alphaSlambdaorder} according to the construction in Definition \ref{definitionalpha}. Now, let $\sigma\colon \{ 
 1, 2, \ldots, |V| \} \to \{ 1, 2, \ldots, |V| \}$ be any bijection. Define the order relation $\blacktriangleleft$ so that 
\begin{equation}\label{vblackvblack}
 v_{\sigma(1)} \blacktriangleleft v_{\sigma(2)} \blacktriangleleft \cdots \blacktriangleleft v_{\sigma(|V|)}
\end{equation}
 and let vertex subsets ordered lexicographically with respect to $\blacktriangleleft$ 
 be determined by rewriting \eqref{vblackvblack} as 
\begin{equation}\label{202949190909195490AM91A}
 w_{1} \blacktriangleleft w_{2} \blacktriangleleft \cdots \blacktriangleleft w_{|V|}, 
\end{equation}
 and by rewriting a vertex subset by replacing a vertex with its index in 
 \eqref{202949190909195490AM91A} (and by then rewriting this resultant set as a word 
 on the alphabet $\{ 1, 2, \ldots \}$). 
 We thus obtain the composition 
\begin{equation}\label{compositionblacktri}
 \left( 
 (\lambda(S'))_{i_{1}^{S, D, \blacktriangleright}}, 
 (\lambda(S'))_{i_{2}^{S, D, \blacktriangleright}}, \ldots, 
 (\lambda(S'))_{i_{\ell(\lambda(S'))}^{S, D, \blacktriangleright}} \right), 
\end{equation}
 according to Definition \ref{definitionalpha}. The case whereby $ \big( i_{j} \rhd^{>} i_{k} \big) \wedge \lnot \big( i_{j} \rhd^{\blacktriangleright} 
 i_{k} \big) $ implies that $ | C_{j}^{S, D} | = | C_{k}^{S, D} |$ (and that the lexicographic ordering of $ C_{j}^{S, D} $ and $ C_{k}^{S, D} $ is 
 reversed if $<$ is used in place of $\blacktriangleleft$ or vice-versa). The symmetric case whereby $ \lnot \big( i_{j} \rhd^{>} i_{k} \big) \wedge \big( 
 i_{j} \rhd^{\blacktriangleright} i_{k} \big) $ again implies that $ | C_{j}^{S, D} | = | C_{k}^{S, D} |$. So, the total orderings $\unrhd^{>}$ and $ 
 \unrhd^{\blacktriangleright}$ \ are equivalent, except for indices $j, k \in \{ 1, 2, \ldots, \ell(\lambda(S')) \}$ for which $ C_{j}^{S, D} $ and 
 $ C_{k}^{S, D} $, i.e., for which $(\lambda(S'))_{j}$ and $(\lambda(S'))_{k}$ are equal, and hence the equality of \eqref{alphaSlambdaorder} 
 and \eqref{compositionblacktri}. 
\end{proof}

\begin{example}
 Letting $D$ and $S$ and $<$ be as in Example \ref{ex2ndcycle}, we again write $V(D) = \{ v_{1} < v_{2} < v_{3} < v_{4} \}$, and we recall that 
 $ \lambda(S') = (2, 1, 1)$ and that the component tuple of $D$ (induced by $S$ and $<$), is 
\begin{equation}\label{v1v3component}
 ( \{ v_{1}, v_{3} \}, \{ v_{4} \}, \{ v_{2} \} ). 
\end{equation}
 Observe that $\text{td} \, \{ v_{1}, v_{3} \} = \text{td} \, \{ v_{4} \} 
 = \text{td} \, \{ v_{2} \} = 0$. 
 We then find that 
\begin{equation}\label{alphanormal211}
 \alpha^{>}_{D}(S) = (2, 1, 1). 
\end{equation}
 Now let us relabel $D$ as below, 
 maintaining the edge highlighting from Example \ref{ex2ndcycle}. 
\begin{center}
\begin{tikzpicture}
\node[vertex] (1) at (0,0) {$w_3$};
\node[vertex] (2) at (2,0) {$w_4$};
\node[vertex] (3) at (0,-2) {$w_1$};
\node[vertex] (4) at (2,-2) {$w_2$};
\draw[black, very thick, edge] (1) -- (2); 
\draw[red, very thick, edge] (1) -- (3);
\draw[black, very thick, edge] (2) -- (4);
\draw[black, very thick, edge] (4) -- (3);
\end{tikzpicture}
\end{center}
 So, we may write 
\begin{align*}
 V(D) & = \{ v_{3} \blacktriangleleft v_{4} \blacktriangleleft v_{1} \blacktriangleleft v_{2} \} \\ 
 & = \{ w_{1} \blacktriangleleft w_{2} \blacktriangleleft w_{3} \blacktriangleleft w_{4} \}. 
\end{align*}
 In this case, the component tuple of $D$ induced by $S$ and $\blacktriangleleft$ is 
\begin{align*}
 & ( \{ w_{1}, w_{3} \}, \{ w_{4} \}, \{ w_{2} \} ) = \\
 & ( \{ v_{3}, v_{1} \}, \{ v_{2} \}, \{ v_{4} \} ), 
\end{align*}
 in contrast to the ordered tuple in \eqref{v1v3component}. 
 We then find that $ \alpha^{\blacktriangleright}_{D}(S) = (2, 1, 1)$, in accordance with \eqref{alphanormal211}. 
\end{example}

 Theorem \ref{theoremalphadefined} gives us that \eqref{displayXD} is well defined in the sense that the expansion in \eqref{displayXD} does not depend 
 on any labeling for $D$. 

\begin{definition}\label{definitionXD}
 For a digraph $D$, we define the \emph{chromatic noncommutative symmetric function} $\text{{\bf X}}_{D}$ so that 
\begin{equation}\label{displayXD}
 \text{{\bf X}}_{D} = \sum_{S \subseteq E(D)} (-1)^{|S|} \Psi_{\alpha(S)}. 
\end{equation}
\end{definition}

\begin{example}
 Letting $D$ denote the digraph shown in Example \ref{examplereversed}, 
 we may verify that the expansion 
\begin{equation}\label{firstXDtoPsiex}
 \text{{\bf X}}_{D} = \Psi_{1111} - 2 \Psi_{211} - \Psi_{121} - \Psi_{112} + 3 \Psi_{31} + \Psi_{13} + 2 \Psi_{22} - 3 \Psi_{4} 
\end{equation}
 holds. Observe how \eqref{firstXDtoPsiex} 
 provides a noncommutative lifting of 
 of the power sum basis expansion 
 that is shown in \eqref{XGtopSageC4} 
 and that is for the chromatic symmetric function indexed 
 by the undirected graph $G$ underlying $D$. In particular, 
 we may verify the relation $\chi(\text{{\bf X}}_{D}) = X_{G}$ 
 with what is produced from the following {\tt SageMath} input. 
\begin{verbatim}
NSym = NonCommutativeSymmetricFunctions(QQ);
Sym = SymmetricFunctions(QQ);
Psi = NSym.Psi();
h = Sym.complete();
p = Sym.power();
p((Psi[1,1,1,1] - 2*Psi[2,1,1] - Psi[1,2,1] - Psi[1,1,2] + 
3*Psi[3,1] + Psi[1,3] + 2*Psi[2,2] - 3*Psi[4]).chi())
\end{verbatim}
 This yields 
\begin{verbatim}
p[1, 1, 1, 1] - 4*p[2, 1, 1] + 2*p[2, 2] + 4*p[3, 1] - 3*p[4]
\end{verbatim}
 and this agrees with the $p$-expansion on display in \eqref{XGtopSageC4}. 
\end{example}

\subsection{Projection onto chromatic symmetric functions}
 Informally, since a directed graph can be thought of as ``containing more information'' compared to the underlying undirected graph, this motivates our 
 application of noncommutative symmetric functions in our devising an analogue for directed graphs of Stanley's chromatic symmetric functions, in the 
 following sense. Since an integer composition can, informally, be thought of as ``containing more information'' compared to the underlying integer 
 partition, this motivates the exploration as to how we could obtain insights on structural properties associated with digraphs by lifting chromatic 
 symmetric functions, which have expansions indexed by partitions, to an algebra with bases indexed by compositions, namely, the algbera 
 $\textsf{NSym}$. This intuitive notion can be formalized with our above construction and with our proof 
 of Theorem \ref{theoremprojection}. 

\begin{theorem}\label{theoremprojection}
 For a digraph $D$ with $G$ as its underlying undirected graph, the projection identity $\chi(\text{{\bf X}}_{D}) = X_{G}$ holds. 
\end{theorem}

\begin{proof}
 From Definition \ref{definitionXD}, we obtain 
\begin{equation}\label{applychisides}
 \chi\left( \text{{\bf X}}_{D} \right) = \sum_{S \subseteq E(D)} (-1)^{|S|} \chi\left( \Psi_{\alpha(S)}\right). 
\end{equation}
 Recalling the product rule in \eqref{Psiprodrule} together with the projection
 property in \eqref{chiPsin} for $\Psi$-generators, we obtain from \eqref{applychisides} that 
\begin{equation}\label{chipsort}
 \chi\left( \text{{\bf X}}_{D} \right) = \sum_{S \subseteq E(D)} (-1)^{|S|} p_{\text{sort}(\alpha(S))}. 
\end{equation}
 Since $\alpha(S)$ is obtained by permuting the entries of the tuple $\lambda(S')$ according to \eqref{alphaSlambdaorder}, where $S'$, again, consists of 
 the edges in $E(G)$ corresponding to the arcs in $S \subseteq E(D)$, we find that the $p$-expansion on the right of \eqref{chipsort} agrees with the 
 power sum expansion of $X_{G}$. 
\end{proof}

\begin{example}
 Letting the digraph 
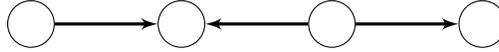
\begin{center}
\begin{tikzpicture}
\node[vertex] (1) at (0,0) {};
\node[vertex] (2) at (2,0) {};
\node[vertex] (3) at (4,0) {};
\node[vertex] (4) at (6,0) {};
\draw[black, very thick, edge] (1) -- (2); 
\draw[black, very thick, edge] (3) -- (2);
\draw[black, very thick, edge] (3) -- (4);
\end{tikzpicture}
\end{center}
 be denoted as $D$, we find that 
 $$ \text{{\bf X}}_{D} = \Psi_{1111} - 2 \Psi_{121} - \Psi_{211} + \Psi_{31} + \Psi_{22} + \Psi_{13} - \Psi_{4}. $$ 
 We may verify that 
 $$ \chi\left( \text{{\bf X}}_{D} \right) = p_{1111} - 3 p_{211} + p_{22} + 2 p_{31} - p_{4} $$ 
 using the following {\tt SageMath} code. 
\begin{verbatim}
NSym = NonCommutativeSymmetricFunctions(QQ);
Psi = NSym.Psi();
Sym = SymmetricFunctions(QQ);
p = Sym.power();
p((Psi[1,1,1,1] - 2*Psi[1,2,1] - Psi[2,1,1] + Psi[3,1] + 
Psi[2,2] + Psi[1,3] - Psi[4]).chi())
\end{verbatim}
 Letting $G$ denote the undirected graph underlying $D$, we may verify Theorem \ref{theoremprojection} using the following {\tt SageMath} input. 
\begin{verbatim}
G = graphs.PathGraph(4);
XG = G.chromatic_symmetric_function(); 
XG
\end{verbatim}
\end{example}

\subsection{Chromatic generating sets}\label{subsectiongenerating} 
 The \emph{chromatic bases for symmetric functions} given by Cho and van Willigenburg \cite{ChovanWilligenburg2016} and reviewed below lead us to 
 consider what would be appropriate as corresponding or analogous bases for $\textsf{NSym}$, as opposed to $\textsf{Sym}$. However, based on the 
 extant literature related to Cho and van Willigenburg's work \cite{ChovanWilligenburg2016}, including 
 \cite{ChovanWilligenburg2018,DahlbergvanWilligenburg2018,HamelHoangTuero2019,Penaguiao2018,Tsujie2018}, it seems that the problem of lifting the 
 Cho--van Willigenburg bases by using digraphs to obtain bases for $\textsf{NSym}$ has not previously been considered. 

 In 2016, Cho and van Willigenburg \cite{ChovanWilligenburg2016} introduced and proved an equivalent 
 formulation of the following analogue of \eqref{Symhgenerators}. 

\begin{theorem}\label{theoremCvW}
 (Cho $\&$ van Willigenburg, 2016) Let $\{ G_{i} \}_{i \geq 1}$ denote a set of connected graphs satisfying the property where $|V(G_{i})| = i$ for every 
 index $i$. Then 
\begin{equation}\label{SymgenXG} 
 \textsf{\emph{Sym}} = \mathbbm{k}[X_{G_{1}}, X_{G_{2}}, \ldots], 
\end{equation}
 and the elements of $\{ X_{G_{k}} \}_{k \geq 1}$ are algebraically independent over the base field \cite{ChovanWilligenburg2016}. 
\end{theorem}

 This leads us to devise an analogue of Theorem \ref{theoremCvW} in the setting of $\textsf{NSym}$, as below. For the purposes of the below proof and 
 of our paper more generally, we let digraphs be without directed loops. 

\begin{theorem}\label{20241003q4q4q1qAqM1az}
 Let $\{ D_{i} \}_{i \geq 1}$ be a set of digraphs such that $|V(D_{i})| = i$ for each $i$ and such that the undirected graph underlying $D_{i}$ is 
 connected for each $i$. Then $$ \textsf{\emph{NSym}} = \mathbbm{k}\langle \text{{\bf X}}_{D_{1}}, \text{{\bf X}}_{D_{2}}, \ldots \rangle, $$ and the 
 elements of $\{ \text{{\bf X}}_{D_{i}} \}_{i \geq 1}$ are algebraically independent over $\mathbbm{k}$. 
\end{theorem}

\begin{proof}
 For the base case such that $i = 1$, we find that $D_{k}$ consists of a single vertex (and does not contain arcs), and we see that 
 $\text{{\bf X}}_{D_{1}} = \Psi_{1}$. Recalling the graded structure for $\textsf{NSym}$ given by the decomposition in 
 \eqref{NSymgraded}, we inductively assume that 
\begin{equation}\label{inductivelyassume}
 \mathbbm{k}\langle \text{{\bf X}}_{D_{1}}, \text{{\bf X}}_{D_{2}}, \ldots, \text{{\bf X}}_{D_{j}} \rangle 
 = \bigoplus_{i = 0}^{\infty} 
 \mathscr{L}\{ \Psi_{\alpha} : \alpha \vDash n, \, \forall k \, \alpha_{k} \leq j \} 
\end{equation}
 Let $G_{j+1}$ denote the undirected graph underlying $D_{j+1}$. Consider the coefficient of $\Psi_{(j+1)}$ in the $\Psi$-expansion of 
 $\text{{\bf X}}_{D_{j+1}}$. We claim that this coefficient is nonzero. We write $\text{{\bf X}}_{D_{j + 1}}$ as a linear combination of expressions 
 of the form $\Psi_{\alpha}$ for $\alpha \vDash j + 1$, and find that the composition $(j+1) \vDash j+1$ is the unique composition of order $\leq j+1$ 
 containing $j + 1$ as an entry. Similarly, with regard to the $p$-expansion of $X_{G_{j+1}}$, we find that partition $(j+1) \vdash j+1$ is the unique 
 partition of order $\leq j+1$ containing $j+1$ as an entry. So, the aforementioned coefficient could not be equal to $0$, because, otherwise, it would 
 not be possible, by Theorem \ref{theoremprojection} and by the graded algebra structure on $\textsf{Sym}$, to express $p_{(j+1)}$ according to 
 the Cho--van Willigenburg generators in \eqref{SymgenXG}, recalling that $G_{j+1}$ is connected (by assumption). From the inductive assumption of 
 \eqref{inductivelyassume}, by taking the free $\mathbbm{k}$-algebra generated by $\{ \text{{\bf X}}_{D_{1}}, \text{{\bf X}}_{D_{2}}, \ldots, 
 \text{{\bf X}}_{D_{j}} \} \cup \{ \text{{\bf X}}_{D_{j+1}} \}$, since the coefficient of $\Psi_{(j+1)}$ in $\text{{\bf X}}_{D_{j+1}}$ is nonzero, we 
 may express $\Psi_{(j+1)}$ by taking a linear combination of $\text{{\bf X}}_{D_{j+1}}$ and a combination of lower-degree generators. That is, 
 $ \mathbbm{k}\langle \text{{\bf X}}_{D_{1}}, \text{{\bf X}}_{D_{2}}, \ldots, \text{{\bf X}}_{D_{j+1}} \rangle = \mathbbm{k}\langle \Psi_{1}, \Psi_{2}, 
 \ldots, \Psi_{j+1} \rangle$, as desired. Since the coefficient of $\Psi_{(j+1)}$ in $\text{{\bf X}}_{j+1}$ is nonzero, and since it is not possible to 
 express the generator $\Psi_{j+1} = \Psi_{(j+1)}$ as a product $\Psi_{\alpha} \Psi_{\beta}$ for nonempty compositions $\alpha$ and $\beta$, we 
 obtain the desired algebraic independence property. 
\end{proof}

\begin{example}
 Cho and van Willigenburg \cite{ChovanWilligenburg2016} provided explicit and combinatorial formulas for generating sets of the form indicated in Theorem 
 \ref{theoremCvW}. This raises questions as to how we could obtain results of a similar nature, according to Theorem \ref{20241003q4q4q1qAqM1az}. 
 In this regard, by letting $S_{n+1}$ denote the star graph of order $(n+1) \geq 1$, Cho and van Willigenburg \cite{ChovanWilligenburg2016} introduced 
 and proved the formula 
\begin{equation}\label{Xstarp}
 X_{S_{n+1}} = \sum_{i=0}^{n} (-1)^{i} \binom{n}{i} p_{(i+1, 1^{n-i})}, 
\end{equation}
 and, since there is a natural directed analogue of $S_{n+1}$, this leads us to formulate a noncommutative version of \eqref{Xstarp}. We define the 
 \emph{inwardly directed star graph} $\text{{\bf S}}_{n}$ as the directed graph $D$ such that $|V(D)| = n$ and such that the following is satisfied. 
 It is possible to label the vertices of $D$ in such a way so that if we write $V(D) = \{ v_{1}, v_{2}, \ldots, v_{n} \}$, then $E(D) = \{ (v_{2}, v_{1}), 
 (v_{3}, v_{1}), \ldots, (v_{n}, v_{1}) \}$. For $n \geq 0$, we can show that 
\begin{equation*}
 \text{{\bf X}}_{\text{{\bf S}}_{n+1}} 
 = \sum_{i=0}^{n} (-1)^{i} \binom{n}{i} \Psi_{(1^{n-i}, i + 1)}, 
\end{equation*}
 following a similar approach as in \cite{ChovanWilligenburg2016}. 
\end{example}

 A combinatorial formula for expanding $X_{P_{n}}$ in the $p$-basis is given in \cite{ChovanWilligenburg2016}, for the path graph $P_{n}$ with a 
 vertex set of size $n$. Devising an analogue of this formula for directed path graphs is complicated by how the total degrees for connected components 
 containing the endpoints could lead to permutations of the tuples in the indices in the $p$-expansion of $X_{P_{n}}$ given in 
 \cite{ChovanWilligenburg2016}. For the time being, we leave this to a future study, and this, in turn, leads us to consider the further areas for future 
 research described in Section \ref{sectionConclusion} below. 

\section{Conclusion}\label{sectionConclusion}
 We conclude with some further areas to explore concerning our lifting to $\textsf{NSym}$
 of chromatic symmetric functions. 

 Letting $\{ e_{\lambda} \}_{\lambda \in \mathcal{P}}$ denote the elementary basis of $\textsf{Sym}$, the $e$-positivity of chromatic symmetric functions 
 is of central interest within the areas of mathematics concerning Stanley's generalization of chromatic polynomials. This is evidenced by the extent of 
 the import of the Stanley--Stembridge conjecture, which would give us every claw-free incomparability grpah is $e$-positive. Following the work of 
 Wang and Wang \cite{WangWang2023chord}, we express that there are many important classes of graphs that have been shown to be $e$-positive, 
 and this includes the classes of complete graphs, paths, and cycles. What classes of digraphs are $E$-positive, for the elementary basis $\{ E_{\alpha} 
 \}_{\alpha \in \mathcal{C}}$ of $\textsf{NSym}$? In view of the importance of Schur-positivity for graphs, this leads us to ask: What classes of 
 digraphs are immaculate-positive, for the immaculate basis $\{ \mathfrak{S}_{\alpha} \}_{\alpha \in \mathcal{C}}$? 

 Observe that the usual Hopf algebra structures on $\textsf{NSym}$ and $\textsf{Sym}$ have not been considered above, at least in direct ways. We 
 encourage a full exploration as to how the Hopf algebra structure on $\textsf{NSym}$ could be applied to develop the theory surrounding our construction 
 of $\text{{\bf X}}_{D}$ for a digraph $D$. How could we obtain cancellation-free and combinatorial formulas for the coproduct $\Delta 
 \text{{\bf X}}_{D}$ for a given digraph $D$, for the usual comultiplication operation $\Delta\colon \textsf{NSym} \otimes \textsf{NSym} \to 
 \textsf{NSym}$ on $\textsf{NSym}$? What is the antipode of $\text{{\bf X}}_{D}$? 

 As a step toward formulating an analogue in the setting of $\textsf{NSym}$ of Stanley's tree isomorphism conjecture, we ask: To what extent can 
 $ \text{{\bf X}}_{D}$, for a directed graph $D$ such that its underlying undirected graph is a tree $T$, determine the isomorphism class for $D$? 

 Given a basis $B$ of $\textsf{NSym}$ formed from expressions of the form $\text{{\bf X}}_{D}$ in the manner described in Section 
 \ref{subsectiongenerating}, how could the dual basis of $\textsf{QSym}$ give light to graph-theoretic properties of $B$ and the digraphs used to 
 construction $B$? How could we obtain combinatorial formulas for products of elements in this dual basis of $\textsf{QSym}$? 

 In view of the positive expansion of $X_{G}$ in the monomial basis, 
 how could we obtain an explicit and cancellation-free 
 expansion for $\text{{\bf X}}_{D}$ 
 in the analogue for $\textsf{NSym}$ introduced by Tevlin \cite{Tevlin2007} 
 of the commutative monomial basis? 

\subsection*{Acknowledgments}
 The author was supported by a Killam Postdoctoral Fellowship from the Killam Trusts.

 \

John M.\ Campbell

Department of Mathematics and Statistics, Dalhousie University

6299 South St., Halifax, NS B3H 4R2, Canada

{\tt jmaxwellcampbell@gmail.com}


\begin{thebibliography}{99}

\bibitem{AlexanderssonPanova2018}
 \textsc{P.\ Alexandersson and G.\ Panova}, 
 L{LT} polynomials, chromatic quasisymmetric functions and graphs with cycles, 
 \emph{Discrete Math.} {\bf 341(12)} (2018), 3453--3482. 

\bibitem{AlexanderssonSulzgruber2021}
 \textsc{P.\ Alexandersson and R.\ Sulzgruber}, 
 {$P$}-partitions and {$p$}-positivity, 
 \emph{Int.\ Math.\ Res.\ Not.\ IMRN} {\bf (14)} (2021), 10848--10907. 

\bibitem{AliniaeifardLivanWilligenburg2024}
 \textsc{F.\ Aliniaeifard, S.\ X.\ Li, and S.\ van Willigenburg}, 
 Generalized chromatic functions, 
 \emph{Int.\ Math.\ Res.\ Not.\ IMRN} {\bf (5)} (2024), 4456--4500. 

\bibitem{AlistePrietoCrewSpirklZamora2021}
 \textsc{J.\ Aliste-Prieto, L.\ Crew, S.\ Spirkl, and J.\ Zamora}, 
 A vertex-weighted {T}utte symmetric function, and constructing graphs with equal chromatic symmetric function, 
 \emph{Electron.\ J.\ Combin.} {\bf 28(2)} (2021), Paper No.\ 2.1, 33. 

\bibitem{BergBergeronSaliolaSerranoZabrocki2014}
 \textsc{C.\ Berg, N.\ Bergeron, F.\ Saliola, L.\ Serrano, and M.\ Zabrocki}, 
 A lift of the {S}chur and {H}all-{L}ittlewood bases to non-commutative symmetric functions, 
 \emph{Canad.\ J.\ Math.} {\bf 66(3)} (2014), 525--565. 

\bibitem{BergeronReutenauerRosasZabrocki2008}
 \textsc{N.\ Bergeron, C.\ Reutenauer, M.\ Rosas, and M.\ Zabrocki}, 
 Invariants and coinvariants of the symmetric groups in noncommuting variables, 
 \emph{Canad.\ J.\ Math.} {\bf 60(2)} (2008), 266--296. 

\bibitem{BessenrodtLuotovanWilligenburg2011}
 \textsc{C.\ Bessenrodt, K.\ Luoto, and S.\ van Willigenburg}, 
 Skew quasisymmetric {S}chur functions and noncommutative {S}chur functions, 
 \emph{Adv.\ Math.} {\bf 226(5)} (2011), 4492--4532. 

\bibitem{Birkhoff191213}
 \textsc{G.\ D.\ Birkhoff}, 
 A determinant formula for the number of ways of coloring a map, 
 \emph{Ann.\ of Math.\ (2)} {\bf 14(1-4)} (1912/13), 42--46. 

\bibitem{CampbellFeldmanLightShuldinerXu2014}
 \textsc{J.\ Campbell, K.\ Feldman, J.\ Light, P.\ Shuldiner, and Y.\ Xu}, 
 A {S}chur-like basis of \textsf{{NS}ym} defined by a {P}ieri rule, 
 \emph{Electron. J. Combin.} {\bf 21(3)} (2014), Paper 3.41, 19. 

\bibitem{ChartrandLesniak2005}
 \textsc{G.\ Chartrand and L.\ Lesniak}, 
 Graphs \& digraphs, 
 Chapman \& Hall/CRC, Boca Raton, FL (2005). 

\bibitem{ChoHuhNam2021}
 \textsc{S.\ Cho, J.\ Huh, and S.-Y.\ Nam}, 
 An analogue of chromatic bases and {$p$}-positivity of skew {S}chur {$Q$}-functions, 
 \emph{J.\ Algebraic Combin.} {\bf 53(4)} (2021), 1037--1056. 

\bibitem{ChovanWilligenburg2016}
 \textsc{S.\ Cho and S.\ van Willigenburg}, 
 Chromatic bases for symmetric functions, 
 \emph{Electron.\ J.\ Combin.} {\bf 23(1)} (2016), Paper 1.15, 7. 

\bibitem{ChovanWilligenburg2018}
 \textsc{S.\ Cho and S.\ van Willigenburg}, 
 Chromatic classical symmetric functions, 
 \emph{J.\ Comb.} {\bf 9(2)} (2018), 401--409. 

\bibitem{CrewHaithcockReynesSpirkl2024}
 \textsc{L.\ Crew, E.\ Haithcock, J.\ Reynes, and S.\ Spirkl}, 
 Homogeneous sets in graphs and a chromatic multisymmetric function, 
 \emph{Adv.\ in Appl.\ Math.} {\bf 158} (2024), Paper No. 102718, 28. 

\bibitem{CrewPechenikSpirkl2023}
 \textsc{L .\ Crew, O.\ Pechenik, and S.\ Spirkl}, 
 The {K}romatic symmetric function: a {$K$}-theoretic analogue of {$X_G$}, 
 \emph{S\'em.\ Lothar.\ Combin.} {\bf 89B} (2023), Art. 77, 12. 

\bibitem{CrewSpirkl2020}
 \textsc{L.\ Crew and S.\ Spirkl}, 
 A deletion-contraction relation for the chromatic symmetric function, 
 \emph{European J.\ Combin.} {\bf 89} (2020), 103143, 20. 

\bibitem{Dahlberg2020}
 \textsc{S.\ Dahlberg}, 
 A new formula for {S}tanley's chromatic symmetric function for unit interval graphs and {$e$}-positivity for triangular ladder graphs, 
 \emph{S\'em.\ Lothar.\ Combin.} {\bf 82B} (2020), Art.\ 59, 12. 

\bibitem{DahlbergvanWilligenburg2020}
 \textsc{S.\ Dahlberg and S.\ van Willigenburg}, 
 Chromatic symmetric functions in noncommuting variables revisited, 
 \emph{Adv.\ in Appl.\ Math.} {\bf 112} (2020), 101942, 25. 

\bibitem{DahlbergvanWilligenburg2018}
 \textsc{S.\ Dahlberg and S.\ van Willigenburg}, 
 Lollipop and lariat symmetric functions, 
 \emph{SIAM J.\ Discrete Math.} {\bf 32(2)} (2018), 1029--1039. 

\bibitem{EaglesFoleyHuangKarangozishviliYu2022}
 \textsc{N.\ M.\ Eagles, A.\ M.\ Foley, 
 A.\ Huang, E.\ Karangozishvili, and A.\ Yu}, 
 {$H$}-chromatic symmetric functions, 
 \emph{Electron. J. Combin.} {\bf 29(1)} (2022), Paper No. 1.28, 39. 

\bibitem{Ellzey2017}
 \textsc{B.\ Ellzey}, 
 Chromatic quasisymmetric functions of directed graphs, 
 \emph{S\'em.\ Lothar.\ Combin.} {\bf 78B} (2017), Art. 74, 12. 

\bibitem{EllzeyWachs2020}
 \textsc{B.\ Ellzey and M.\ L.\ Wachs}, 
 On enumerators of {S}mirnov words by descents and cyclic descents, 
 \emph{J. Comb.} {\bf 11(3)} (2020), 413--456. 

\bibitem{GebhardSagan2001}
 \textsc{D.\ D.\ Gebhard and B.\ E.\ Sagan}, 
 A chromatic symmetric function in noncommuting variables, 
 \emph{J. Algebraic Combin.} {\bf 13(3)} (2001), 227--255. 

\bibitem{GelfandKrobLascouxLeclercRetakhThibon1995}
 \textsc{I.\ M.\ Gelfand, D.\ Krob, A.\ Lascoux, B.\ Leclerc, V.\ S.\ Retakh, and J.-Y.\ Thibon}, 
 Noncommutative symmetric functions, 
 \emph{Adv. Math.} {\bf 112(2)} (1995), 218--348. 

\bibitem{HamelHoangTuero2019}
 \textsc{A.\ M.\ Hamel, C.\ T.\ Ho\`ang, and J.\ E.\ Tuero}, 
 Chromatic symmetric functions and {$H$}-free graphs, 
 \emph{Graphs Combin.} {\bf 35(4)} (2019), 815--825. 

\bibitem{LoehrWarrington2024}
 \textsc{N.\ A.\ Loehr and G.\ S.\ Warrington}, 
 A rooted variant of {S}tanley's chromatic symmetric function, 
 \emph{Discrete Math.} {\bf 347(3)} (2024), Paper No.\ 113805, 18. 

\bibitem{Macdonald1995}
 \textsc{I.\ G.\ Macdonald}, 
 Symmetric functions and {H}all polynomials, 
 The Clarendon Press, Oxford University Press, New York (1995). 

\bibitem{NenashevaZhukov2021}
 \textsc{M.\ Nenasheva and V.\ Zhukov}, 
 An extension of {S}tanley's chromatic symmetric function to binary delta-matroids, 
 \emph{Discrete Math.} {\bf 344(11)} (2021), Paper No. 112549, 10. 

\bibitem{Penaguiao2018}
 \textsc{R.\ Penaguiao}, 
 The kernel of chromatic quasisymmetric functions on graphs and nestohedra, 
 \emph{S\'em.\ Lothar.\ Combin.} {\bf 80B} (2018), Art.\ 10, 12. 

\bibitem{ShareshianWachs2016}
 \textsc{J.\ Shareshian and M.\ L.\ Wachs}, 
 Chromatic quasisymmetric functions, 
 \emph{Adv. Math.} {\bf 295} (2016), 497--551. 

\bibitem{Stanley1995}
 \textsc{R.\ P.\ Stanley}, 
 A symmetric function generalization of the chromatic polynomial of a graph, 
 \emph{Adv. Math.} {\bf 111(1)} (1995), 166--194. 

\bibitem{Tevlin2007}
 \textsc{L.\ Tevlin}, 
 Noncommutative monomial symmetric functions, 
 \emph{Proc.\ FPSAC'07, Tianjin, China} (2007). 

\bibitem{Tsujie2018}
 \textsc{S.\ Tsujie}, 
 The chromatic symmetric functions of trivially perfect graphs and cographs, 
 \emph{Graphs Combin.} {\bf 34(5)} (2018), 1037--1048. 

\bibitem{WangWang2023chord}
 \textsc{D.\ G.\ L.\ Wang and M.\ M.\ Y.\ Wang}, 
 The {$e$}-positivity of two classes of cycle-chord graphs, 
 \emph{J. Algebraic Combin.} {\bf 57(2)} (2023), 495--514. 

\bibitem{WangWang2020}
 \textsc{D.\ G.\ L.\ Wang and M.\ M.\ Y.\ Wang}, 
 A combinatorial formula for the {S}chur coefficients of chromatic symmetric functions, 
 \emph{Discrete Appl. Math.} {\bf 285} (2020), 621--630. 

\bibitem{WangZhouunpublished}
 \textsc{D.\ G. L. Wang and J.\ Z.\ F.\ Zhou}, 
 Composition method for chromatic symmetric functions: Neat noncommutative analogs, 
 arXiv:2401.01027 (2024). 

\bibitem{White2021}
 \textsc{J.\ A.\ White}, 
 The chromatic quasisymmetric class function of a digraph, 
 \emph{Ann.\ Comb.} {\bf 25(4)} (2021), 961--993. 

\bibitem{Wolf1936}
 \textsc{M.\ C.\ Wolf}, 
 Symmetric functions of non-commutative elements, 
 \emph{Duke Math. J.} {\bf 2(4)} (1936), 626--637. 

\end{thebibliography}
\end{document}